\documentclass[12pt]{amsart}
\usepackage{amsfonts, amstext}
\usepackage[usenames,dvipsnames]{xcolor}	
\usepackage{amsmath, amsthm, amssymb, amsfonts}						
\usepackage{fullpage}
\usepackage{enumitem}
\usepackage{mathrsfs}
\usepackage{tikz-cd}
\usepackage{hhline}
\usepackage{textcomp}
\usepackage{mathbbol}

\usepackage{hyperref}

\newcommand{\mcal}{\mathcal}
\newcommand{\mscr}{\mathscr}
\newcommand*{\sHom}{\mscr{H}\kern -.5pt om}
\newcommand*{\sExt}{\mscr{E}\kern -.5pt xt}

\newcommand{\PP}{\mathbb{P}}

\newcommand{\Z}{\mathbb{Z}}

\renewcommand{\O}{\mcal{O}}

\DeclareMathOperator{\Pic}{Pic\,}

\DeclareMathOperator{\codim}{codim\,}

\DeclareMathOperator{\rk}{rk}
\DeclareMathOperator{\adj}{adj}
\DeclareMathOperator{\sing}{sing}

\DeclareMathOperator{\Bl}{Bl}

\newtheorem{theorem}{Theorem}[section]

\newtheorem{proposition}[theorem]{Proposition}
\newtheorem{corollary}[theorem]{Corollary}
\newtheorem{lemma}[theorem]{Lemma}

\theoremstyle{definition}

\newtheorem{definition}[theorem]{Definition}

\theoremstyle{remark}

\title{Tetragonal Canonical Curves}
\author{Henry Fontana}
\address{Department of Mathematics, Statistics and Computer Science, University of Illinois at Chicago, Chicago, IL, 60607}
\email{hfonta2@uic.edu}

\begin{document}

\begin{abstract}
In this note we compute the Harder-Narasimhan filtration of $N_{C/\PP^{g-1}}$ where $C$ is a general tetragonal canonical curve $C \subset \PP^{g-1}$ over an algebraically closed field $k$. Any tetragonal canonical curve $C$ lies on a $3$-dimensional rational normal scroll $Q \subset \PP^{g-1}$. Due to a corollary of a result of Schreyer we know that $C$ is a complete intersection in $Q$, and thus the normal bundle $N_{C/Q}$ splits as a direct sum. Using a result of Coşkun and Smith we show that the quotient $N_{Q/\PP^{g-1}} \rvert_C$ is stable for a general tetragonal canonical curve $C$, this gives us the HN-filtration of $N_{C/\PP^{g-1}}$.
\end{abstract}

\maketitle

\section{Introduction}

Let $C$ be a smooth non-hyperelliptic curve $C$ of genus $g$ over an algebraically closed field.  The normal bundle $N_{C/\PP^{g-1}}$ is an intrinsic invariant which controls the deformations of $C$ in its canonical embedding $C \hookrightarrow \PP^{g-1}$. In \cite{AFO16} Aprodu, Farkas, and Ortega conjectured that $N_{C/\PP^{g-1}}$ is semi-stable for a general canonical curve. This conjecture was confirmed by Coşkun, Larson, and Vogt in \cite{CLV23} where they proved that $N_{C/\PP^{g-1}}$ is semi-stable for the general curve when $g \geq 7$. However, the normal bundle of a trigonal canonical curve $C \subset \PP^{g-1}$ is unstable, due to the fact that such a curve lies on a surface scroll $S$. In \cite{FON25} the author shows that $N_{C/S} \subset N_{C/\PP^{g-1}}$ gives the Harder-Narasimhan filtration for a general trigonal canonical curve $C$. A tetragonal canonical curve $C$ lies on a threefold scroll $Q \subset \PP^{g-1}$ and $N_{C/Q} \subset N_{C/\PP^{g-1}}$ is a destabilizing subbundle. Using a result of Schreyer, section $6$ of \cite{Sch86}, we can compute the normal bundle $N_{C/Q}$ of the curve in the scroll.

\begin{lemma}
Let $C$ be a tetragonal canonical curve of genus $g$ and $Q$ the $3$-fold scroll on which it lies. Denote by $H$ the embedding class of $Q \subset \PP^{g-1}$ and $R$ the class of a fiber. Then $C$ is a complete intersection in $Q$ of surfaces $Y,Z$ such that
$$[Y]=2H-b_1R$$
$$[Z]=2H-b_2R$$
where $b_1+b_2=g-5$
\end{lemma}

The numbers $b_1$ and $b_2$ are invariants of $C$ determined by the graded-betti numbers, by convention we assume $b_1 \geq b_2$. Furthermore in \cite{BP21} Bujokas and Patel prove that $\lvert b_1 - b_2 \rvert \leq 1$ for the general tetragonal canonical curve. In the case when $g$ is even we must have $b_1 > b_2$ which implies $N_{C/Y}$ is a destabilizing line subbundle of $N_{C/Q}$. When $g$ is odd we have $b_1=b_2$ for the general curve so that $N_{C/Q}$ is semi-stable. The main result of this note is the computation of the HN-filtration of $N_{C/\PP^{g-1}}$ where $C$ is tetragonal of genus $g$.

\begin{theorem} \label{T1}
Let $C$ be a general tetragonal canonical curve of genus $g \geq 7$ with $g \equiv 1 \mod 2$, then
$$0 \subset N_{C/Q} \subset N_{C/\PP^{g-1}}$$
is the Harder-Narasimhan filtration of $N_{C/\PP^{g-1}}$.
\end{theorem}

\begin{theorem} \label{T2}
Let $C$ be a general tetragonal canonical curve of genus $g \geq 6$ with $g \equiv 0 \mod 2$, then
$$0 \subset N_{C/Y} \subset N_{C/\PP^{g-1}}$$
is the Harder-Narasimhan filtration of $N_{C/\PP^{g-1}}$.
\end{theorem}

In the next section we introduce preliminary results and conventions. In section $3$ we will use degeneration to prove that $N_{Q/\PP^{g-1}} \rvert_C$ is stable for the general curve, this proves Theorem \ref{T1}. We also obtain the stability of the bundle $N_{Q/\PP^{g-1}}$ as a corollary.

\begin{corollary}
Suppose that
$$Q = \PP(\O_{\PP^1} \oplus \O_{\PP^1}(a) \oplus \O_{\PP^1}(b))$$
where $0 \leq a,b \leq 1$. Let $H$ denote the class of $\O_{\PP(\mathcal{E})}(1)$ while $R$ is the class of a fiber. Then for any $c > 0$ if $r=3c+a+b+2$ the linear series $\lvert H + cR \rvert$ embeds $Q$ as a smooth rational normal scroll in $\PP^r$ such that $N_{Q/\PP^r}$ is stable.
\end{corollary}

\begin{proof}
Set $g:=3c+a+b+3=r+1$ and observe that $\lvert E+cR \rvert$ defines an isomorphism of $Q$ onto the projective bundle
$$\PP(\O_{\PP^1}(c) \oplus \O_{\PP^1}(a+c) \oplus \O_{\PP^1}(b+c)) \subset \PP^r$$
Therefore the class $4H^2+(8c-2g+10)HR$ contains smooth tetragonal canonical curves of genus $g$ whose canonical class is the restriction to $C$ of $E+cR$. It follows that $N_{Q/\PP^r} \rvert_C$ is stable so we conclude that $N_{Q/\PP^r}$ is also stable, since the restriction of a destabilizing subbundle $\mathcal{F} \subset N_{Q/\PP^r}$ would destabilize $N_{Q/\PP^r} \rvert_C$.
\end{proof}

Furthermore since $N_{Q/\PP^{g-1}} \rvert_C$ is a quotient of $N_{Y/\PP^{g-1}} \rvert_C$ the proof of Theorem \ref{T1} gives a bound on the slope of subbundles $\mathcal{F} \subset N_{Y/\PP^{g-1}} \rvert_C$. It turns out that this bound is already good enough to give the semi-stability of $N_{Y/\PP^{g-1}} \rvert_C$, so the results of section $3$ will also prove Theorem \ref{T2}.

\section{Preliminaries}

\subsection{Canonical Curves}

If $C$ is a non-hyperelliptic algebraic curve of genus $g$ the canonical linear series defines an embedding $C \hookrightarrow \PP^{g-1}$. Conversely if $C \subset \PP^{g-1}$ is smooth of genus $g$ and degree $2g-2$ then $O_C(1) \cong \omega_C$. We therefore call a smooth genus $g$ and degree $2g-2$ curve in $\PP^{g-1}$ a \textbf{canonical curve}. Canonical curves lie in an irreducible component of the Hilbert scheme the general canonical curves are those lying in some open subset of this component.

\subsection{Vector Bundles}
Given a  vector bundle $\mathcal{F}$ on an algebraic curve $C$ we define the slope to be
$$\mu(\mathcal{F})=\frac{\deg(\mathcal{F})}{\rk(\mathcal{F})}$$
A vector bundle $\mathcal{E}$ is \textbf{semi-stable} if $\mu(\mathcal{F}) \leq \mu(\mathcal{E})$ for all subbundles $\mathcal{F} \subset \mathcal{E}$, and is stable if each inequality is strict. Every vector bundle admits a \textbf{Harder-Narasimhan} filtration in which successive quotients are semi-stable.
\begin{proposition}
If $\mathcal{F}$ is a vector bundle on an algebraic curve $C$ then there is a unique filtration
$$0=\mathcal{E}_0 \subset \mathcal{E}_1 \subset \dots \subset \mathcal{E}_k=\mathcal{F}$$
where the $\mathcal{E}_i$ are subbundles such that
\begin{itemize}
\item The quotient $\mathcal{G}_i=\mathcal{E}_{i}/\mathcal{E}_{i-1}$ is semi-stable for $i=1,\dots,k$
\item The slopes of the $\mathcal{G}_i$ are decreasing
$$\mu(\mathcal{G}_1) > \mu(\mathcal{G}_2) > \dots > \mu(\mathcal{G}_k)$$
\end{itemize}
\end{proposition}

If $C \subset \PP^r$ is a subscheme with ideal sheaf $\mathcal{I}_C$ then the conormal sheaf is
$$N_C^{\vee} \cong \frac{\mathcal{I}_C}{\mathcal{I}_C^2}$$
When $C$ has at worst nodal singularities then $N_C^{\vee}$ is locally free. If $C$ is smooth there is a short exact sequence of vector bundles

\[\begin{tikzcd}
	0 & {T_C} & {T_{\PP^r} \rvert_C} & {N_C} & 0
	\arrow[from=1-1, to=1-2]
	\arrow[from=1-2, to=1-3]
	\arrow[from=1-3, to=1-4]
	\arrow[from=1-4, to=1-5]
\end{tikzcd}\]

In particular when $r=g-1$ and $C$ is a canonical curve the above exact sequence implies
$$\deg(N_C)=\deg(T_{\PP^r} \rvert_C)-\deg(T_C)=2g^2-2$$
Since $\rk(N_C)=g-2$ we obtain the following Proposition. 
\begin{proposition} \label{P6}
If $N_C$ is the normal bundle of a smooth genus $g$ canonical curve $C \subset \PP^{g-1}$ then
$$\mu(N_C)=2g+4+\frac{6}{g-2}$$
\end{proposition}
The normal bundle is easy to understand for complete intersections
$$X=D_1 \cap \dots \cap D_k$$
In this case given any point $p \in C$ we have
$$T_p X=T_p D_1 \cap T_p D_2 \cap \dots \cap T_p D_k$$
so there is an isomorphism
$$N_{D_1} \rvert_C \oplus \dots \oplus N_{D_k} \rvert_C \to N_C$$

\begin{proposition} \label{P1}
If $C \subset X$ is a complete intersection of divisors $D_1,\dots,D_k$ then
$$N_{C/X} \cong \O_C(D_1) \oplus \dots \oplus \O_C(D_k)$$
\end{proposition}

The following result from section $3$ of \cite{CS25} is useful in degeneration arguments. 

\begin{proposition}[\cite{CS25}] \label{P3}
If $X \subset \PP^d$ is an elliptic normal curve with $d \geq 3$ then $N_{X/\PP^d}$ is semi-stable.
\end{proposition}

To end this subsection we prove a lemma that allows us to bound the slope of subbundles of the quotient of a semi-stable bundle.

\begin{lemma} \label{L3}
Suppose that we have a short exact sequence of bundles
\[\begin{tikzcd}
	0 & {\mathcal{F}} & {\mathcal{E}} & {\mathcal{Q}} & 0
	\arrow[from=1-1, to=1-2]
	\arrow[from=1-2, to=1-3]
	\arrow[from=1-3, to=1-4]
	\arrow[from=1-4, to=1-5]
\end{tikzcd}\]
where $\mathcal{E}$ is semi-stable. Then $\mu(\mathcal{M}) \leq \mu(\mathcal{E})$ for every subbundle $\mathcal{M} \subset \mathcal{Q}$.
\end{lemma}

\begin{proof}
The quotient bundle $\mathcal{M}'$ of an inclusion $\mathcal{M} \subset \mathcal{Q}$ is also a quotient of $\mathcal{E}$. Since $\mathcal{E}$ is semistable we have $\mu(\mathcal{M}') \geq \mu(\mathcal{E})$ which implies that
$$\rk(\mathcal{E})\deg(\mathcal{M}) \leq \deg(\mathcal{E})\rk(\mathcal{M})-(\deg(\mathcal{E})\rk(\mathcal{Q})-\rk(\mathcal{E})\deg(\mathcal{Q})) \leq  \deg(\mathcal{E})\rk(\mathcal{M})$$
\end{proof} 

\subsection{The adjusted slope}
The adjusted slope generalizes $\mu$ to vector bundles on connected reducible nodal curves. Let
$$X=X_1 \cup \dots X_k$$ 
be a connected union where the $X_i$ are smooth curves. We denote by $\nu:\tilde{X} \to X$ the normalization map and $p_1,p_2$ the points in the normalization over $p \in X_{\sing}$. Given any vector bundle $\mathcal{E}$ on $X$ there is a canonical isomorphism of the fibers ${\nu^* \mathcal{E}}_{p_1} \to \nu^* \mathcal{E}_{p_2} \cong \mathcal{E}_p$. Therefore given a subbundle $\mathcal{F} \subseteq \nu^* \mathcal{E}$ the intersection $\mathcal{F}_{p_1} \cap \mathcal{F}_{p_2}$ makes sense as a subspace of $\mathcal{E}_p$. We denote by $\codim_{\mathcal{F}}(\mathcal{F}_{p_1} \cap \mathcal{F}_{p_2})$ the codimension of this intersection in either $\mathcal{F}_{p_1}$ or $\mathcal{F}_{p_2}$, note that these codimensions are equal because $\dim(\mathcal{F}_{p_1})=\dim(\mathcal{F}_{p_2})$. With this convention the adjusted slope of $\mathcal{F}$ is defined as follows.

\begin{definition}
Let $X$ be a connected curve with only nodes as singularities. The adjusted slope of a subbundle $\mathcal{F} \subset \mathcal{M}=\nu^* \mathcal{E}$ is
$$\mu^{\adj}(\mathcal{F})=\mu(\mathcal{F}) - \frac{1}{\rk(\mathcal{F})} \sum_{p \in X_{\sing}} \codim_{\mathcal{F}}(\mathcal{F}_{p_1} \cap \mathcal{F}_{p_2})$$
\end{definition}

A result from \cite{CLV23} allows us to reduce semi-stability for a bundle $\mathcal{E}$ on a general curve to semi-stability of $\mathcal{E}$ on a specific nodal curve. 

\begin{proposition}[\cite{CLV23}] \label{P2}
Suppose $\mathscr{C} \to \Delta$ is a family of connected nodal curves over the spectrum of a discrete valuation ring and $\mathcal{E}$ a vector bundle on $\mathscr{C}$. Let $\mathcal{E}_0$ and $\mathcal{E}_t$ denote the restriction to the special and general fibers respectively. If $\mu^{\adj}(\mathcal{F}_0) \leq c$ for all subbundles of $\mathcal{E}_0$ then $\mu^{\adj}(\mathcal{F}_t) \leq c$ for all subbundles of $\mathcal{E}_t$.
\end{proposition}

\subsection{Tetragonal Curves}

Let $C \subset \PP^{g-1}$ be a canonical curve of genus $g$ which admits a linear series $g^1_d$ of dimension $1$ and degree $d$. If $D \in g^1_d$ is an effective divisor then by Riemann-Roch we have
$$h^0(\omega_C(-D))=g-d+1$$
Thus there is a $(g-1)-(d-2)$ dimensional space of hyperplanes in $\PP^{g-1}$ containing $D$. This implies that the points of $D$ span a linear space of dimension $d-2$. In other words if $\phi:C \to \PP^1$ is the map induced by the $g^1_d$ then the fiber $\phi^{-1}(t)$ over any point spans a $d-2$ dimensional linear space in the canonical embedding of $C$. As $t$ varies over $\PP^1$ these linear spaces sweep out a rational normal scroll $Q \subset \PP^{g-1}$ containing $C$.

The scroll $Q$ also admits an algebraic description. Let $V$ denote the cokernel of the inclusion of vector bundles $\O_{\PP^1} \hookrightarrow \phi_* \O_C$. All vector bundles on $\PP^1$ split so we can write
$$V \cong \O_{\PP^1}(a_1) \oplus \dots \oplus \O_{\PP^1}(a_{d-1})$$
We have $\sum a_i=g-d+1$ and the relative canonical maps $C$ into the projective bundle $\PP(V) \cong Q$. A result of Ballico \cite{Bal89} shows that $V$ is balanced for the general tetragonal curve $C$, meaning that $\lvert a_i -a_j \rvert \leq 1$ for all $i,j$. In particular if $g \geq 2d-2$ then we must have $a_i > 0$ for all $i$ i.e. $Q \subset \PP^{g-1}$ is smooth of degree $g-d+1$. In \cite{Sch86} Schreyer computed a resolution for the ideal sheaf $\mathscr{I}_C$ of $C$ in $Q$.
\begin{proposition} \label{P5} \cite[Theorem 2.5]{Sch86}
$C \subset Q$ has a resolution whose first graded piece has the form
\[\begin{tikzcd}
	\dots & {\bigoplus_{i=1}^{\beta}\O_Q(-2H+b_iR)} & {\mathscr{I}_C} & 0
	\arrow[from=1-1, to=1-2]
	\arrow[from=1-2, to=1-3]
	\arrow[from=1-3, to=1-4]
\end{tikzcd}\]
where $\beta=d(d-3)/2$
\end{proposition}
If we push forward the free module in the proposition under $\phi$ we get
$$N_1 = \bigoplus_{i=1}^{\beta} \O_{\PP^1}(b_i)$$
which is referred to as the first syzygy bundle. In \cite{BP21} Bujokas and Patel prove balancedness of the first syzygy bundle.
\begin{proposition}[\cite{BP21}] \label{P4}
For the general tetragonal curve $C$ the first syzygy bundle is balanced.
\end{proposition}

Proposition \ref{P5} has a particularly nice form when $C$ is tetragonal. Note that $\beta=2$ when $d=4$ so Proposition \ref{P5} says that $C$ is a complete intersection.
\begin{corollary}
Let $C$ be a tetragonal canonical curve of genus $g \geq 6$ and $Q$ the $3$-fold scroll on which it lies. Then $C$ is a complete intersection in $Q$ of divisors $Y,Z$
$$[Y]=2H-b_1R$$
$$[Z]=2H-b_2R$$
where $H$ is the hyperplane class, $R$ is the fiber class, and $b_1+b_2=g-5$.
\end{corollary}
We assume WLOG that $b_1 \geq b_2$. Combining the above Corollary with Proposition \ref{P1} we obtain a description of $N_{C/Q}$.
\begin{corollary} \label{C1}
If $C$ is a tetragonal canonical curve and $Q$ is scroll on which it lies then
$$N_{C/Q} \cong \O_C(2H-b_1R) \oplus \O_C(2H-b_2R)$$
\end{corollary}
Using the Corollary we see that
$$\deg(N_{C/Q})=[C] \cdot (4H-(g-5)R)=4(2g-2)-4(g-5)=4g+12$$
thus $\mu(N_{C/Q}) > \mu(N_{C/\PP^{g-1}})$ by Proposition \ref{P6} so $N_{C/Q}$ is a destabilizing subbundle.
Our goal is to show that $N_{Q/\PP^{g-1}} \rvert_C$ is semi-stable for the general curve. Note that we have
$$\deg(N_{Q/\PP^{g-1}} \rvert_C)=\deg(N_{C/\PP^{g-1}})-\deg(N_{C/Q})=$$
$$(2g^2-2)-(4g+12)=2g^2-4g-14$$
from which we find the slope
\begin{equation} \label{EQ1}
\mu(\mathcal{E})=\frac{2g^2-4g-14}{g-4}=2g+4+\frac{2}{g-4}
\end{equation}

By Corollary \ref{C1} together with Propositions \ref{P1} and \ref{P4} we get that $N_{C/Q}$ is semi-stable for a general tetragonal curve of odd genus. For even genus $N_{C/Q}$ is unstable with HN-filtration
$$0 \subset \O_C\left(2H - b_2R \right) \subset N_{C/Q}$$
If $C$ is general of even genus then $b_2=(g-6)/2$ by Proposition \ref{P4}, we compute
$$\deg(N_{Y/\PP^{g-1}} \rvert_C)=\deg(N_{C/\PP^{g-1}})-\deg(N_{C/Y})=$$
$$2g^2-2-[Y] \cdot [C]=2g^2-2g-10$$
this gives the slope of $N_{Y/\PP^{g-1}} \rvert_C$ for a general tetragonal curve of even degree
\begin{equation} \label{EQ2}
\mu(N_{Y/\PP^{g-1}} \rvert_C)=2g+4+\frac{2}{g-3}
\end{equation}

We will also need the following lemma regarding rational normal surface scrolls, a proof can be found in \cite{FON25}.

\begin{lemma}{\cite{FON25}} \label{L2}
Let $W \subset \PP^n$ be a minimal degree nondegenerate surface scroll (i.e. a ruled surface over $\PP^1$) and $Y$ a rational normal curve of degree $k$ with $2 \leq k \leq n-1$. If there exists a linear space $\PP^k \subset \PP^n$ with $Y \subset W \cap \PP^k$ then $Y= W \cap \PP^k$.
\end{lemma}

\section{Stability of the Quotient}

We now show stability of $N_{Q/\PP^{g-1}} \rvert_C$ where $C$ is a general tetragonal canonical curve of genus $g$ and $Q$ is the $3$-fold scroll on which $C$ lies. The strategy is to degenerate $C$ to a $g$-secant union $C_1 \cup C_2$ where $C_1 \subset \PP^{g-2} \subset \PP^{g-1}$ is a rational normal curve of degree $g-2$ and $C_2 \subset Q \subset \PP^{g-1}$ is an elliptic normal curve such that
$$[C_1]=H^2+HR$$
$$[C_2]=3H^2-(2g-9)HR$$
We take $C_1$ to be a hyperplane section of a divisor $S \subset Q$ of class $H+R$. Observe that $C_1$ is rational of degree $g-2$ since $(H^2+HR) \cdot R=1$ and $(H^2+HR) \cdot H=g-2$.

\begin{lemma}
If $C$ is a general tetragonal curve and $Q$ is the scroll containing it then the class
$$[C_2]=3H^2-(2g-9)HR$$
on $Q$ contains elliptic curves.
\end{lemma}

\begin{proof}
We have $Q \cong \PP(\mathcal{E})$ where
$$\mathcal{E} = \O_{\PP^1}(a_1) \oplus \O_{\PP^1}(a_2) \oplus \O_{\PP^1}(a_3)$$
and since $C$ is general $\mathcal{E}$ is balanced. Then there is an isomorphism $\phi$ from $Q$ to one of the following depending on the value of $g \mod 3$
\begin{itemize}
\item $\PP^1 \times \PP^2$ when $g=0 \mod 3$.
\item $\Bl_L \PP^3$, the blow up of $\PP^3$ at a line $L$, when $g=1 \mod 3$.
\item The small resolution of the cone over a smooth quadric $\PP^1 \times \PP^1 \subset \PP^3 \subset \PP^4$ with vertex a point $p$, when $g=2 \mod 3$.
\end{itemize}

We have $\Pic(Q) \cong \Z H \oplus \Z R$ and the isomorphism $\phi$ acts on the Picard group as
$$F \mapsto F \ \ \ H \mapsto H-\left\lfloor \frac{g-3}{3} \right\rfloor F$$
with this correspondence the following are elliptic curves with class $[C_2]$.

\begin{itemize}
\item When $Q \cong \PP^1 \times \PP^2$ take a smooth cubic curve in $\PP^2$ and $\psi$ a degree three map from this curve to $\PP^1$, then the graph of $\psi$ is a curve of class $[C_2]$.
\item When $Q \cong \Bl_L \PP^3$ take an elliptic normal curve meeting a line $L$ in one point, then its strict transform in $\Bl_L \PP^3$ has class $[C_2]$.
\item If $g= 2 \mod 2$ let $X \subset \PP^4$ be an elliptic normal curve, and $Y$ the projection of $X$ from a point $p \in X$ onto a hyperplane $H \subset \PP^4$. Then $Y$ lies on a smooth quadric hypersurface $Q \subset H$.
\end{itemize}
\end{proof}

Let $\nu:X \to C_1 \cup C_2$ be the normalization map and define 
$$\mathcal{E}:=\nu^* N_{Q/\PP^{g-1}} \rvert_{C_1 \cup C_2}$$ 
Note that we have isomorphisms
$$\mathcal{E} \rvert_{C_1} \cong N_{Q/\PP^{g-1}} \rvert_{C_1}$$
$$\mathcal{E} \rvert_{C_2} \cong N_{Q/\PP^{g-1}} \rvert_{C_2}$$
There is a commutative diagram
\[\begin{tikzcd}
	0 & {T_{C_1}} & {T_Q \rvert_{C_1}} & {N_{C_1/Q}} & 0 \\
	0 & {T_{\PP^{g-2}} \rvert_{C_1}} & {T_{\PP^{g-1}} \rvert_{C_1}} & {\O_{C_1}(1)} & 0
	\arrow[from=1-1, to=1-2]
	\arrow[from=1-2, to=1-3]
	\arrow[from=1-2, to=2-2]
	\arrow[from=1-3, to=1-4]
	\arrow[from=1-3, to=2-3]
	\arrow[from=1-4, to=1-5]
	\arrow[from=1-4, to=2-4]
	\arrow[from=2-1, to=2-2]
	\arrow[from=2-2, to=2-3]
	\arrow[from=2-3, to=2-4]
	\arrow[from=2-4, to=2-5]
\end{tikzcd}\]

I claim that the map $N_{C_1/Q} \to \O_{C_1}(1)$ is surjective. If not then there is a point $p \in C_1$ such that $T_{Q,p} \subset T_{\PP^{g-2},p}$. This implies that $L \subset \PP^{g-2}$ where $L$ is the fiber of the rational scroll $S$ containing $p$, which contradicts lemma $\ref{L2}$. It follows from the snake lemma that we have a surjective map
$$N_{C_1/\PP^{g-2}} \twoheadrightarrow \mathcal{E} \rvert_{C_1}$$
Since $N_{C_1/\PP^{g-2}}$ is stable with slope $g$ by lemma \ref{L3} we conclude that given any subbundle $\mathcal{F} \subset \mathcal{E}$ we have $\mu(\mathcal{F} \rvert_{C_1}) \leq g$.

There is also a surjective map of bundles
$$N_{C_2/\PP^{g-1}} \twoheadrightarrow \mathcal{E} \rvert_{C_2}$$
Since $C_2$ is an elliptic normal curve $N_{C_2/\PP^{g-1}}$ is semi-stable by Proposition \ref{P3} and its slope is
$$\mu(N_{C_2/\PP^{g-1}})=\frac{g^2}{g-2}=g+2+\frac{4}{g-2}$$
It follows that if $\mathcal{F} \subset \mathcal{E}$ is a subbundle then $\mu(\mathcal{F} \rvert_{C_2}) \leq g+2+4/(g-2)$. Combining this with the above we get
$$\mu^{\adj}(\mathcal{F}) \leq \mu(\mathcal{F})=\mu(\mathcal{F} \rvert_{C_1})+\mu(\mathcal{F} \rvert_{C_2}) \leq 2g+2+\frac{4}{g-2}$$
Thus by Proposition \ref{P2} we obtain the following Lemma
\begin{lemma} \label{L1}
If $C$ is a general tetragonal curve then 
$$\mu(\mathcal{F}) \leq 2g+2+\frac{4}{g-2}$$ 
for every subbundle $\mathcal{F} \subset N_{Q/\PP^{g-1}} \rvert_C$
\end{lemma}
By equation \ref{EQ1} we have
$$\mu(\mathcal{E})=2g+4+\frac{2}{g-4}$$
so that Theorem \ref{T1} follows directly from Lemma \ref{L1}. Furthermore there is a surjection $N_{Y/\PP^{g-1}} \rvert_C \twoheadrightarrow N_{Q/\PP^{g-1}} \rvert_C$ so $\mu(\mathcal{F}) \leq 2g+2+4/(g-2)$ for every subbundle $\mathcal{F} \subset N_{Y/\PP^{g-1}} \rvert_C$, this proves Theorem \ref{T2}.

\bibliographystyle{amsalpha}
\bibliography{refs}
\nocite{*}

\end{document}